\documentclass[11 pt]{article}
\usepackage{latexsym}
\usepackage{graphicx}
\usepackage{amsmath}
\usepackage{amsthm}
\usepackage{amssymb,amsfonts,enumerate}
\usepackage[pdfpagelabels=true,plainpages=false,colorlinks=true,linkcolor=blue,citecolor=blue,urlcolor=blue]{hyperref}
\newtheorem{thm}{Theorem}[section]

\newtheorem{df}{Definition}[section]
\newtheorem{remark}{Remark}[section]
\newtheorem{exmp}{Example}[section]

\begin{document}
\begin{center}
{\huge Certain properties of bounded variation of sequences of fuzzy numbers by using generalized weighted mean}\\\vspace{.4cm}
Sarita Ojha* and P. D. Srivastava\\
Department of Mathematics, Indian Institute of Technology Kharagpur,\\ Kharagpur-721302, India\\
Email : sarita.ojha89@gmail.com
\end{center}

\section*{Abstract}
The class of bounded variation $bv^F(u,v)$ of fuzzy numbers introduced by \cite{SO1} has been investigated further with the help of the generalized weighted mean
matrix $G(u,v)$. Imposing some restrictions on the matrix $G(u,v)$, we have established it's relation with different class of sequences such as our known classical 
sets, set of all statistically null difference sequences, Cesaro sequences etc. Also we have examined the concepts like equivalent fuzzy number, symmetric fuzzy 
number on this quasilinear space.\\\\
{\bf Keywords:} Sequence of fuzzy number, Bounded variation, Generalized weighted mean, Equivalent fuzzy number.\\
{\bf AMS subject classification:} 46S40, 03E72.

\section{Introduction and preliminaries}
Since the concept of fuzzy numbers, introduced by Zadeh (1965), several mathematicians have studied extensively from different aspects of
its theory and applications such as fuzzy analysis, fuzzy topology, fuzzy measure, fuzzy decision making etc. Motivated by this, many
authors \cite{B}, \cite{S}, \cite{SO} have introduced new class of sequences of fuzzy numbers and investigated convergence and other topological properties.\\
Since the set of all fuzzy numbers can be embedded in $\mathbb{R}$, one may think that the results proved in reals is a particular case of fuzzy
numbers. But as every fuzzy number does not have an inverse element with respect to addition i.e. does not form a group structure, so most of the facts
known for reals may not be valid in fuzzy setting. Therefore this theory is not a mere extension of the results hold in $\mathbb{R}$.\\
A fuzzy real number $X:R \to [0,1]$ is a fuzzy set which is is normal, fuzzy convex, upper semi-continuous and the support
$X^0=\overline{\{t\in R:X(t)>0\}}$ is compact.
Clearly, $R$ is embedded in $L(R)$, the set of all fuzzy numbers, in this way: for each $r\in R, \ \overline{r}\in L(R)$ is defined as,
\begin{center}
$\overline{r}(t) = \left\{
\begin{array}{c l}
  1, & t=r \\
  0, & t\neq r
\end{array}
\right.$
\end{center}
For, $0<\alpha\leq1$, $\alpha$-cut of a fuzzy number $X$ is defined by $X^{(\alpha)}=\{t\in R:X(t)\geq\alpha\}$. The set $X^{(\alpha)}$ is a closed, bounded and non-empty interval for each $\alpha\in [0,1]$. For any two fuzzy numbers $X,Y$, Matloka \cite{M} proved that $L(R)$ is a complete under the following metric,
\begin{eqnarray*}
d(X,Y)=\sup\limits_{0\leq \alpha \leq1} \max\{|\underline{X}^{(\alpha)}-\underline{Y}^{(\alpha)}|,|\overline{X}^{(\alpha)}-\overline{Y}^{(\alpha)}|\}
\end{eqnarray*}
where $\underline{X}^{(\alpha)}$ and $\overline{X}^{(\alpha)}$ are the lower and upper bound of the $\alpha$-cut.\\

\begin{thm} \cite{TB} (Representation Theorem) Let $X ^{(\alpha)} = [\underline{X} ^{(\alpha)}, \overline{X} ^{(\alpha)}]$ for $u\in L(R)$ and for
each $\alpha\in[0, 1]$. Then the following statements hold:
\begin{enumerate}[(i)]
\item $\underline{X} ^{(\alpha)}$ is a bounded and non-decreasing left continuous function on (0, 1].
\item $\overline{X} ^{(\alpha)}$ is a bounded and non-increasing left continuous function on (0, 1].
\item The functions $\underline{X} ^{(\alpha)}$ and $\overline{X} ^{(\alpha)}$ are right continuous at the point $\alpha=0$.
\item $\underline{X} ^{(1)} \leq \overline{X} ^{(1)}$.
\end{enumerate}
Conversely, if the pair of functions $P$ and $Q$ satisfies the conditions (i)-(iv), then
there exists a unique $X\in L(R)$ such that $X ^{(\alpha)} = [P(\alpha), Q(\alpha)]$ for each $\alpha\in [0, 1]$.
The fuzzy number $X$ corresponding to the pair of functions $P$ and $Q$ is defined by
$X : R\to[0, 1],\ X(t) = \sup\{\alpha : P(\alpha)\leq t \leq Q(\alpha)\}$.
\end{thm}

\noindent Now we list some basic definitions of fuzzy numbers as discussed in \cite{Q}, \cite{Q1} from the algebraic point of view.
\begin{df}
\begin{enumerate}[(i)]
\item A fuzzy number $S\in L(R)$ is said to be symmetric if $S(t)=S(-t)$ for all $t\in\mathbb{R}$ i.e. $S=-S$.
\item For any two fuzzy numbers $X,Y$, $X$ is said to be equivalent to $Y$ or $X\sim Y$ if and only if there exists two symmetric fuzzy numbers $S_1,S_2$ such that
\begin{equation*}
X+S_1=Y+S_2
\end{equation*}
\item The midpoint function $X_M:\mathbb{R}\to [0,1]$ of a fuzzy number $X$ is defined by assigning the mid point of each $\alpha$-level sets to $X_M ^{(\alpha)}$ for all $\alpha\in [0,1]$ i.e.
    \begin{equation*}
X_M ^{(\alpha)}=\frac{\underline{X}^{(\alpha)}+\overline{X}^{(\alpha)}}{2}
\end{equation*}
\end{enumerate}
\end{df}

\noindent Sequences of bounded variation for fuzzy numbers has been investigated by many authors \cite{TB1}, \cite{TD}, \cite{TD1}. The most recent generalization
in this direction is to define bounded variation by using the generalized weighted mean. In 2011, Polat et al \cite{PKS} have taken the definition of
generalized weighted mean as : let $U$ be the set of all real sequences $u=(u_n)$ such that $u_n\neq 0$ for all $n\in \mathbb{N}$. Then consider
\begin{center}
$G(u,v)=g_{nk} = \left\{
\begin{array}{c l}
  u_kv_n, & 0\leq k\leq n \\
  0, & k>n.
\end{array}
\right.$
\end{center}
for all $k,n\in\mathbb{N}$, where $u_n$ depends only on $n$ and $v_k$ only on $k$. The matrix $G(u,v)$, defined above, is called as generalized
weighted mean or factorable matrix. Clearly the $k$-th row sum of $G(u,v)$ is $\sum\limits_{i=0} ^k u_kv_i$.\\
Motivated by their work, in 2015, Ojha and Srivastava \cite{SO1} defined the class $bv^F (u,v)$ of sequences of fuzzy numbers by using the generalized mean
matrix $G(u,v)$ as:
\begin{eqnarray*}
bv^F (u,v)=\Big\{X=(X_k)\in w^F:\sum\limits_{k=0}^{\infty}\Big|\sum\limits_{i=0}^k u_kv_id(\Delta X_i,\bar{0})\Big|<\infty\Big\}
\end{eqnarray*}
and established many topological properties of it. The main aim of this paper is to investigate some interesting relation of the above quasilinear space
with other known classical sets of sequences of fuzzy numbers by imposing conditions on $u,v$ or on the matrix $G(u,v)$.

\section{Main results}
\begin{thm}
\begin{enumerate}[(i)]
\item The set $l_p ^F\subset bv^F(u,v)$ if $v=(v_i)\in l_q$ and $u=(u_i)\in l_1$ for $1<p< \infty$ where $\frac{1}{p}+\frac{1}{q}=1$.
\item For $p=\infty$, $l_{\infty} ^F\subset bv^F(u,v)$ if the row sum of $G(u,v)$ is in $l_1$.\\
\end{enumerate}
\noindent The inclusion is strict in sense in both cases.
\end{thm}
\begin{proof}
We shall give the proof for $1<p<\infty$. Let $(X_k)\in l_p ^F$ and $(v_i)\in l_q$. Then $\exists\ M>0$ such that $\sum d(X_k,\bar{0})^p=M$. Now,
\begin{eqnarray*}
\Big|\sum\limits_{i=0} ^k u_kv_id(\Delta X_i,\bar{0})\Big| &\leq& \Big|\sum\limits_{i=0} ^k u_kv_id(X_i,\bar{0})\Big|+\Big|\sum\limits_{i=0} ^k u_kv_i
d(X_{i+1},\bar{0})\Big|\\
&\leq& |u_k|\sum\limits_{i=0} ^k| v_id(X_i,\bar{0})|+|u_k|\sum\limits_{i=0} ^k| v_id(X_{i+1},\bar{0})|\\
&\leq& |u_k|\Big[(\sum\limits_{i=0} ^k |v_i|^q)^{1/q}(\sum\limits_{i=0} ^k d(X_i,\bar{0})^p)^{1/p}+(\sum\limits_{i=0} ^k |v_i|^q)^{1/q}
(\sum\limits_{i=0} ^k d(X_{i+1},\bar{0})^p)^{1/p}\Big]\\
&\leq& |u_k|(\sum\limits_{i=0} ^k |v_i|^q)^{1/q}2M^{1/p}
\end{eqnarray*}
So $(X_k)\in bv^F(u,v)$ if $(u_i)\in l_1$.\\
The case $p=\infty$ follows almost the same lines.. So we omit it.\\
To prove that the inclusion is strict, let us consider the following example.
\begin{exmp}
Let $p=\infty$ and $u_k=\frac{1}{k^4},v_k=1$ for all $k$. Clearly in this case, $q=1$.\\
\begin{equation*}
\sum\limits_{i=0} ^k u_kv_i=\frac{1}{k^4}\sum\limits_{i=0} ^k 1=\frac{1}{k^3}\in l_1.
\end{equation*}
Now define $(X_k)$ as
\begin{center}
$X_k = \left\{
\begin{array}{c l}
  \bar{k}, & k\ \mbox{odd} \\
  \bar{0}, & k\ \mbox{even}
\end{array}
\right.$\end{center}
And so
\begin{center}
$d(\Delta X_k,\bar{0})=\left\{
\begin{array}{c l}
  k, & k\ \mbox{odd} \\
  k+1, & k\ \mbox{even}
\end{array}
\right.$\end{center}
Without any loss of generality take $k$ as odd. Then
\begin{eqnarray*}
v_1d(\Delta X_1,\bar{0})+v_3d(\Delta X_3,\bar{0})+\cdots+v_kd(\Delta X_k,\bar{0}) &=& 1+3+\cdots+k=k^2\\
\mbox{and }\ v_0d(\Delta X_0,\bar{0})+v_2d(\Delta X_2,\bar{0})+\cdots+v_{k-1}d(\Delta X_{k-1},\bar{0}) &=& 1+3+\cdots+k=k^2
\end{eqnarray*}
\begin{eqnarray*}
\sum\limits_{i=0} ^k u_kv_i d(\Delta X_i,\bar{0}) &=& \mbox{sum at all odd points}+\mbox{sum at all even points}\\
&=& \frac{1}{k^4}2k^2=\frac{2}{k^2}\in l_1.
\end{eqnarray*}
So $(X_k)\notin l_{\infty} ^F$ but $(X_k)\in bv^F (u,v)$.
\end{exmp}
\end{proof}

\begin{remark}
We can also prove similar result i.e. $l_p ^F (\Delta)\subset bv^F(u,v)$ by taking the same conditions.
\end{remark}

\begin{thm}
The limit of a sequence in $c^F(\Delta)$ is preserved under the metric in $bv^F(u,v)$ if the row sum of the matrix $G(u,v)$ are in $l_1$.
\end{thm}
\begin{proof}
Let $X^n=(X^n _i)$ be a sequence in $c^F(\Delta)$ converging to $X=(X_i)$. Then under the metric defined by Savas \cite{S}, we have
\begin{eqnarray}
&& \rho(X^n,X)\to 0\ \mbox{as}\ n\to\infty\nonumber\\
\mbox{i.e.} && d(X^n _0,X_0)+\sup\limits_{i} d(\Delta X^n _i,\Delta X_i)\to 0\ \mbox{as}\ n\to\infty\nonumber\\
\mbox{So} && d(X^n _0,X_0)\to 0\ \mbox{and}\ \sup\limits_{i} d(\Delta X^n _i,\Delta X_i)\to 0\ \mbox{as}\ n\to\infty
\end{eqnarray}
Also since the row sum of the matrix $G(u,v)$ are in $l_1$, so there exists $M>0$ such that $\sum\limits_{k=0} ^{\infty} |\sum\limits_{i=0} ^k u_kv_i|=M$
and $\lim\limits_{k\to\infty}\big( \sum\limits_{i=0} ^k u_kv_i\big)=0$.
Now the metric in $bv^F(u,v)$ as in \cite{SO1},
\begin{eqnarray*}
D(X^n,X)= |u_0v_0d(X^n _0,X_0)|+\sum\limits_{k=0} ^{\infty} |\sum\limits_{i=0} ^k u_kv_id(\Delta X^n _i,\Delta X_i)|
\end{eqnarray*}
Since $(u_i),(v_i)\in U$, so
\begin{eqnarray*}
&& u_0v_0d(X^n _0,X_0)\to 0\ \mbox{and from (1), since} \sup\limits_{i} d(\Delta X^n _i,\Delta X_i)\to 0\ \mbox{as}\ n\to\infty\\
\mbox{so,} && d(\Delta X^n _i,\Delta X_i)\to 0\ \mbox{as}\ n\to\infty\ \mbox{for each}\ i\\
\mbox{i.e.}&& \sum\limits_{i=0} ^k u_kv_id(\Delta X^n _i,\Delta X_i)\to 0
\end{eqnarray*} Then
\begin{eqnarray*}
\sum\limits_{k=0} ^{\infty} |\sum\limits_{i=0} ^k u_kv_id(\Delta X^n _i,\Delta X_i)| &\leq& \sup\limits_{i} d(\Delta X^n _i,\Delta X_i)\sum\limits_{k=0} ^{\infty}
\big|\sum\limits_{i=0} ^k u_kv_i\big|\\
&=& M \sup\limits_{i} d(\Delta X^n _i,\Delta X_i)\to 0\ \mbox{as}\ n\to\infty \ \mbox{follows from (1)}
\end{eqnarray*}
So we can conclude $D(X^n,X)\to 0$ as $n\to\infty$. So, the limit is same in $bv^F(u,v)$. This completes the proof.
\end{proof}

\begin{thm}
If the infimum of the row sum of $G(u,v)$ exists and greater than 0, then $bv^F(u,v)\subset S_0 ^F(\Delta)$ where $S_0 ^F(\Delta)$ is the
set of all statistically null difference sequences of fuzzy numbers \cite{B}.
\end{thm}
\begin{proof}
Let $(X_k)\in bv^F(u,v)$, then $\sum\limits_{k=0} ^{\infty} \Big|\sum\limits_{i=0} ^k u_kv_id(\Delta X_i,\bar{0})\Big|=M$, (say). Now for any $\varepsilon>0$
and $n\in\mathbb{N}$,
\begin{eqnarray}
M=\sum\limits_{k=0} ^{n} \Big|\sum\limits_{i=0} ^k u_kv_id(\Delta X_i,\bar{0})\Big|\geq \varepsilon
\Big|\Big\{k\leq n: \Big|\sum\limits_{i=0} ^k u_kv_id(\Delta X_i,\bar{0})\Big|\geq\varepsilon\Big\}\Big|
\end{eqnarray}
In the R.H.S $|\{\cdot\}|$ means the cardinality of the enclosed set. Since the infimum of the row sum of the matrix $G(u,v)$ exists, so take
$\inf\limits_{k}|\sum\limits_{i=0} ^k u_kv_i|=m$ Then $m>0$ by given condition.
\begin{eqnarray}
\Big|\sum\limits_{i=0} ^k u_kv_id(\Delta X_i,\bar{0})\Big| &\geq& m\sum\limits_{i=0} ^k d(\Delta X_i,\bar{0})\nonumber\\
\mbox{So } \Big\{k\leq n: m\sum\limits_{i=0} ^k d(\Delta X_i,\bar{0})\geq\varepsilon\Big\} &\subseteq&
\Big\{k\leq n: \Big|\sum\limits_{i=0} ^k u_kv_id(\Delta X_i,\bar{0})\Big|\geq\varepsilon\Big\}\nonumber \\
\mbox{i.e. } \Big\{k\leq n: \sum\limits_{i=0} ^k d(\Delta X_i,\bar{0})\geq\frac{\varepsilon}{m}\Big\} &\subseteq&
\Big\{k\leq n: \Big|\sum\limits_{i=0} ^k u_kv_id(\Delta X_i,\bar{0})\Big|\geq\varepsilon\Big\}
\end{eqnarray}
Also, $\sum\limits_{i=0} ^k d(\Delta X_i,\bar{0})\geq d(\Delta X_k,\bar{0})$ for each $k$. Thus we get
\begin{eqnarray}
\Big\{k\leq n: d(\Delta X_i,\bar{0})\geq\frac{\varepsilon}{m}\Big\} \subseteq \Big\{k\leq n: \sum\limits_{i=0} ^k d(\Delta X_i,\bar{0})\geq
\frac{\varepsilon}{m}\Big\}
\end{eqnarray}
Combining (3) and (4) we get
\begin{eqnarray*}
\Big|\Big\{k\leq n: d(\Delta X_i,\bar{0})\geq\frac{\varepsilon}{m}\Big\}\Big|\leq
\Big|\Big\{k\leq n: \Big|\sum\limits_{i=0} ^k u_kv_id(\Delta X_i,\bar{0})\Big|\geq\varepsilon\Big\}\Big|
\end{eqnarray*}
Thus from (2), we get
\begin{eqnarray*}
\varepsilon \Big|\Big\{k\leq n: d(\Delta X_i,\bar{0})\geq\frac{\varepsilon}{m}\Big\}\Big| &\leq& M\\
\varepsilon\cdot \frac{1}{n}\Big|\Big\{k\leq n: d(\Delta X_i,\bar{0})\geq\frac{\varepsilon}{m}\Big\}\Big| &\leq& \frac{M}{n}\to 0\ \mbox{as}\ n\to\infty.
\end{eqnarray*}
Since $\varepsilon>0$ is arbitrary, so $\Delta X_n \xrightarrow{st} 0$ as $n\to\infty$. This completes the proof.
\end{proof}

\begin{thm}
Let $(X_k)$ be a Cesaro sequence of fuzzy numbers converges to some $L\in L(R)$, then $(X_k)\in bv^F(u,v)$ if the row sum of the matrix $G(u',v)$ are in $l_1$
where $u'=(u' _k)=ku_k$.
\end{thm}
\begin{proof}
Since $(X_k)$ is a Cesaro sequence converges to $L$, so for some integer $k_0$ such that
\begin{eqnarray*}
\frac{1}{k}\sum\limits_{i=0} ^k d(X_i,L) &\leq& \varepsilon\ \forall\ k\geq k_0\\
\mbox{So } \frac{1}{k}\sum\limits_{i=0} ^k d(X_i,\bar{0}) &\leq& \frac{1}{k}\sum\limits_{i=0} ^k [ d(X_i,L)+d(L,\bar{0})]\\
&\leq& \varepsilon+d(L,\bar{0})\ \mbox{whenever}\ k\geq k_0\\
\mbox{i.e. } \sum\limits_{i=0} ^k d(X_i,\bar{0}) &\leq& k[\varepsilon+d(L,\bar{0})]\ \mbox{whenever}\ k\geq k_0
\end{eqnarray*}
Now for all $k\geq k_0$
\begin{eqnarray*}
\Big|\sum\limits_{i=0} ^k u_kv_id(\Delta X_i,\bar{0})\Big| &\leq& \Big|\sum\limits_{i=0} ^k u_kv_id(X_i,\bar{0})\Big|
+\Big|\sum\limits_{i=0} ^k u_kv_id(X_{i+1},\bar{0})\Big|\\
&\leq& |u_k|\Big|\sum\limits_{i=0} ^k v_ik(\varepsilon+d(L,\bar{0}))+\sum\limits_{i=0} ^k v_ik(\varepsilon+d(L,\bar{0}))\Big|\\
&\leq& (\varepsilon+d(L,\bar{0}))|u_k|\Big|\sum\limits_{i=0} ^k v_i2k\Big|\\
&\leq& 2(\varepsilon+d(L,\bar{0}))\ |ku_k|\Big|\sum\limits_{i=0} ^k v_i\Big|
\end{eqnarray*}
Now considering, $u' _k=ku_k$, we get our required result.
\end{proof}

\begin{thm}
Let $(X_i),(Y_i)$ be two sequences of fuzzy numbers such that $X_i\sim Y_i$ i.e. $X_i+S_i=Y_i+S'_i$ where $S_i,S'_i$ are symmetric fuzzy numbers. If
$S_i,S'_i\in l^F _{\infty}$ and the row sum of $G(u,v)$ are in $l_1$, then $(X_i)\in bv^F(u,v) \Leftrightarrow (Y_i)\in bv^F(u,v)$.
\end{thm}
\begin{proof}
It is enough to prove it for one-sided implication. Let $(X_i)\in bv^F(u,v)$. Since $(S_i)\in l^F _{\infty}$,
so $(S_i)\in bv^F(u,v)$ as the row sum of $G(u,v)$ are
in $l_1$. \\
$\Rightarrow \ (X_i+S_i)\in bv^F(u,v)$ as $bv^F(u,v)$ is closed under addition.\\
$\Rightarrow \ (Y_i+S' _i)\in bv^F(u,v)$ since $X_i+S_i=Y_i+S'_i$.\\
Now we only have to show that $(Y_i)\in bv^F(u,v)$.
\begin{eqnarray*}
|\sum\limits_{i=0} ^k u_kv_i d(\Delta Y_i,\bar{0})| &\leq& |\sum\limits_{i=0} ^k u_kv_i d(\Delta Y_i,\Delta (Y_i+S' _i))|+|\sum\limits_{i=0} ^k u_kv_i
d(\Delta (Y_i+S' _i),\bar{0})|\\
&\leq& |\sum\limits_{i=0} ^k u_kv_i d(\bar{0},\Delta S' _i)|+|\sum\limits_{i=0} ^k u_kv_i d(\Delta (Y_i+S' _i),\bar{0})|\\
&& \mbox{Since $d(X+Z,Y+Z)=d(X,Y)$ where $X,Y,Z\in L(R)$}\\
&<& \infty
\end{eqnarray*}
$\Rightarrow \ (Y_i)\in bv^F(u,v)$. The converse implication follows the same lines. This proves the theorem.
\end{proof}

\begin{thm}
Let $(X_{M,i})$ be the sequence of midpoint function of the of the sequence of fuzzy numbers $(X_i)$ for each $i$. Then
$(X_i)\in bv^F(u,v)\implies (X_{M,i})\in bv^F(u,v)$.
\end{thm}
\begin{proof}
For any fuzzy number $X$, we know $X_M ^{(\alpha)}=\frac{\underline{X} ^{(\alpha)}+\overline{X} ^{(\alpha)}}{2}$. Also,
$d(X,\bar{0})=\max\{\underline{X} ^0,\overline{X} ^0\}$
\begin{eqnarray}
d(X_M,\bar{0}) &=& \Big|\frac{\underline{X} ^0+\overline{X} ^0}{2}\Big|\ \mbox{As the upper and lower cut are same for $X_M$}\nonumber\\
&\leq& \max\{\underline{X} ^0,\overline{X} ^0\} =d(X,\bar{0})
\end{eqnarray}
Now for a sequence of midpoint fuzzy numbers $(X_{M,i})$ of $(X_i)$, we can have
\begin{eqnarray*}
\Delta X_{M,i} &=& X_{M,i}-X_{M,i+1}\\
(\Delta X_{M,i})^{(\alpha)} &=& (X_{M,i})^{(\alpha)}-(X_{M,i+1})^{(\alpha)}\\
&=& \frac{\underline{X_i} ^{(\alpha)}+\overline{X_i} ^{(\alpha)}}{2} - \frac{\underline{X_{i+1}} ^{(\alpha)}+\overline{X_{i+1}} ^{(\alpha)}}{2}\\
&=& \frac{\underline{X_i} ^{(\alpha)}-\overline{X_{i+1}} ^{(\alpha)}}{2}+\frac{\overline{X_i} ^{(\alpha)}-\underline{X_{i+1}} ^{(\alpha)}}{2}\\
&=& \frac{\underline{(\Delta X_i)} ^{(\alpha)}}{2}+ \frac{\overline{(\Delta X_i)} ^{(\alpha)}}{2}\\
&=& (\Delta X_i)_M
\end{eqnarray*} and so
\begin{eqnarray*}
d(\Delta X_{M,i},\bar{0}) &=& d((\Delta X_i)_M,\bar{0})\\
\Big|\sum\limits_{i=0} ^k u_kv_i d(\Delta X_{M,i},\bar{0})\Big| &=& \Big|\sum\limits_{i=0} ^k u_kv_i d((\Delta X_i)_M,\bar{0})\Big|\\
&\leq& \Big|\sum\limits_{i=0} ^k u_kv_i d(\Delta X_i,\bar{0})\Big| \ \mbox{Follows from (5)}
\end{eqnarray*}
From the above inequality it is clear that if $(X_i)\in bv^F(u,v)$, then $(X_{M,i})\in bv^F(u,v)$.
\end{proof}

\begin{remark}
The converse of the above is not true in general i.e. $(X_{M,i})\in bv^F(u,v)$ for some $(X_i)$ does not imply $(X_i)\in bv^F(u,v)$. To show this, let us
take $(X_k)$ be the triangular fuzzy number $[-k,0,k]$. Then clearly, for each $\alpha\in [0,1]$, we have $X_{M,k}=\bar{0}$ for all $k$. So for any choice of
$(u,v)$, $(X_{M,i})\in bv^F(u,v)$. Whereas, $\Delta X_k=X_k-X_{k+1}=[-k,0,k]-[-k-1,0,k+1]=[-2k-1,0,2k+1]$ and therefore $d(\Delta X_k,\bar{0})=2k+1$ which
does not belong to $bv^F(u,v)$ for some suitable choice of $(u,v)$.\\
We already know that equivalent fuzzy numbers have same mid point. But due to the fact that sequence of midpoint fuzzy numbers belong to $bv^F(u,v)$ does
not confirm $(X_i)\in bv^F(u,v)$, so, the above two theorems can not be equivalent.
\end{remark}

\end{document}